\documentclass{article}[12pt]
\usepackage{bbm}
\usepackage{amsmath}
\usepackage{amsthm}
\usepackage{amssymb}
\usepackage{amsfonts}
\usepackage{euscript}
\usepackage{amscd}
\usepackage{graphicx}
\usepackage{color}
\usepackage{epsfig}
\usepackage{tikz}
\usepackage{latexsym}
\usepackage{fancyhdr}
\usepackage[all,cmtip]{xy}
\usepackage{srcltx}
\usepackage{titletoc}
\usepackage[bf]{titlesec}
\usepackage{mathrsfs}
\usepackage{verbatim}
\usepackage{textcomp}
\usepackage{enumerate}
\usepackage{makeidx}
\usepackage{multirow}
\usepackage{multicol}
\usepackage{cite}
\usepackage{makeidx}
\usepackage{enumerate}
\usepackage{soul}
\usepackage{enumitem}
\usepackage{xfrac}
\usepackage{appendix}

\newcommand{\PSL}{PSL(2,\C)}

\newcommand{\R}{\mathbb{R}}
\newcommand{\Z}{\mathbb{Z}}
\newcommand{\N}{\mathbb{N}}
\newcommand{\C}{\mathbb{C}}
\newcommand{\CC}{\ensuremath{{\widehat{\mathbb C}}}}
\newcommand{\CP}{\ensuremath{{\mathbb{CP}}}}

\newcommand{\racs}[1]{\Omega^1(-1^{#1})}
\newcommand{\RI}[1]{\mathcal{RI}\Omega^1(-#1)}
\newcommand{\Sim}[1]{ \mathcal{S}\left( #1 \right)}

\renewcommand{\Re}[1]{\mathfrak{R}e\left(#1\right)}
\renewcommand{\Im}[1]{\mathfrak{I}m\left(#1\right)}
\newcommand{\rp}[1]{\left\langle #1 \right\rangle}

\newtheorem{definition}{\textbf{Definition}}[section]
\newtheorem{lemma}[definition]{\textbf{Lemma}}
\newtheorem{proposition}[definition]{\textbf{Proposition}}
\newtheorem{corollary}[definition]{\textbf{Corollary}}
\newtheorem{theorem}[definition]{\textbf{Theorem}}
\newtheorem{remark}[definition]{\textbf{Remark}}
\newtheorem{example}[definition]{\textbf{Example}}

\makeindex
\newtheoremstyle{TheoremNum}
        {\topsep}{\topsep}              
        {\itshape}                      
        {}                              
        {\bfseries}                     
        {.}                             
        { }                             
        {\thmname{#1}\thmnote{ \bfseries #3}}
    \theoremstyle{TheoremNum}
    \newtheorem{rtheorem}{Theorem}
    \newtheorem{rproposition}{Proposition}
    \newtheorem{rcorollary}{Corollary}
    
\begin{document}

\title{Classification of rational 1--forms on the Riemann sphere up to $PSL(2, \mathbb{C})$}


\author{Julio C. \textsc{Maga\~na--C\'aceres} \\
\small{Centro de Ciencias Matem\'aticas UNAM,} \\
\small{Campus Morelia C.P. 58190,}  \\
\small{Instituto de F\'isica y Matem\'aticas UMSNH,}\\
\small{C.P. 58040 Morelia, Michoac\'an M\'exico} \\
\small{julio@matmor.unam.mx}   \\
\small{ORCID: 0000--0002--3272--7541}
}

\maketitle

\begin{abstract}
We study the family $\racs{s}$ of rational 1--forms on the Riemann sphere, 
having exactly $-s \leq -2$ simple poles. 
Three equivalent $(2s-1)$--dimensional complex atlases on $\racs{s}$, 
using coefficients, 
zeros--poles and residues--poles of the 1--forms, 
are recognized. 
A rational 1--form are called isochronous when all their residues are purely imaginary. 
We prove that the subfamily $\RI{s}$ of isochronous 1--forms is a $(3s-1)$--dimensional real analytic submanifold in the complex manifold $\racs{s}$. 
The complex Lie group $\PSL$ acts holomorphically on $\racs{s}$. 
For $s \geq 3$, 
the $\PSL$--action is proper on $\racs{s}$ and $\RI{s}$. 
Therefore, 
the quotients $\racs{s}/\PSL$ and $\RI{s}/\PSL$ admit a stratification by orbit types. 
Realizations for $\racs{s}/\PSL$ and $\RI{s}/\PSL$ are given, 
using an explicit set of $\PSL$--invariant functions.

\noindent \textbf{Classification:}{ Primary 32M05; Secondary 30F30, 58D19.}\\
\noindent \textbf{Keywords:}{ Rational 1--forms, isochronous centers, proper $\PSL$--action, principal $\PSL$--bundle, stratified space by orbit types.}
\end{abstract}


\section{Introduction}

A compact Riemann surface $M_g$ of genus $g\geq 0$ is determined by its space of holomorphic 1--forms. 
For $g \geq 1$, the Jacobian variety $J(M_g)$ is defined by using the vector space of holomorphic 1--forms. 
Very roughly speaking, Torelli's theorem says that we can recover $M_g$ from $J(M_g)$; 
see \cite[p.~359]{Griffith}. 
Moreover, for $g \geq 2$, 
a classical result of Hurwitz asserts that the automorphism group $Aut(M_g)$ is finite; 
see \cite[Ch. V]{farkas}. 
Looking at the moduli space of compact Riemann surfaces $\mathcal{M}_{g,0}$, 
generically $Aut(M_g)$ is trivial. 
For $g=0$, three special features appear. 
Obviously there exists only one complex structure, set the Riemann sphere $\CC$. 
Secondly, 
any holomorphic 1--form over $\CC$ is identically zero. 
Finally, 
the automorphism group $\PSL$ of $\CC$ is the biggest between the automorphism groups of any Riemann surface. \vspace{0cm}
\begin{flushleft}
\textit{A natural problem is the classification of rational 1--forms on $\CC$ with $-s \leq -2$ poles up to the automorphism group $\PSL$.}
\end{flushleft}\vspace{0cm}
The family of rational 1--forms on $\CC$ is an infinite dimensional vector space. 
A natural idea is to consider the stratification by the multiplicities of zeros and poles. 
Let $\racs{s}$ be the family of rational 1--forms having exactly $-s \leq -2$ simple poles.
Our techniques naturally allows us to study rational 1--forms with zeros of any multiplicity and simple poles. 
We recognized three equivalent $(2s-1)$--dimensional complex atlases on $\racs{s}$, 
looking at different expressions of the rational 1--forms coefficients $\Omega^1_{coef}(-1^s)$, zeros--poles $\Omega^1_{zp}(-1^s)$ and residues--poles $\Omega^1_{rp}(-1^s)$. 
Our main result is the following theorem:
\begin{rtheorem}[\ref{teorema-equivalencia-parametros}]
  The complex manifolds $\Omega_{coef}^{1}(-1^s)$, $\Omega_{zp}^{1}(-1^s)$ and $\Omega_{rp}^{1}(-1^s)$ are biholomorphic.
\end{rtheorem}
For the proof, 
we construct explicitly the biholomorphisms $\Omega_{rp}^1(1^s) \longrightarrow \Omega_{coef}^1(1^s)$ and $\Omega_{zp}^1(1^s) \longrightarrow \Omega_{coef}^1(1^s)$ using the \textit{Vi\`ete} map defined in \eqref{viete-map}.  

The group $\PSL$ acts naturally on $\racs{s}$ by coordinate changes. 
As a result, 
this $\PSL$--action is proper for $s \geq 3$; 
see Lemma \ref{proper-action}.
By using the classical theory of proper Lie group actions as in \cite[Ch.~2]{diustermaat}, 
we recognize a principal $\PSL$--bundle $\pi_s: \mathcal{G}(-1^s) \to \mathcal{G}(-1^s)/\PSL$, 
where $\mathcal{G}(-1^s) \subset \racs{s}$ is the open and dense subset of generic rational 1--forms with trivial isotropy group in $\PSL$, 
and $\pi_s$ denotes the natural projection to the orbit space. 
For $s=3 \ (\text{resp. } s=4)$, we prove that the principal $\PSL$--bundle is trivial (resp. nontrivial). 
On the other hand, 
for all $\omega \in \racs{s}\setminus \mathcal{G}(-1^s)$, 
the isotropy group $\PSL_{\omega}$ is a nontrivial finite subgroup of $\PSL$. 
A classical result of F. Klein classifies the finite subgroups of $\PSL$; see \cite[p. 126]{klein20}. 
In our framework, 
the realization problem is which finite subgroups of $\PSL$ are realizable as the isotropy groups of suitable $\omega \in \racs{s}$? 
A positive answer is as follows.
\begin{rproposition}[\ref{realizacion-grupos-finitos}]
 Every finite subgroup $G < \PSL$ appears as the isotropy group of suitable $\omega \in \racs{s}$.
\end{rproposition}
\noindent Obviously, 
the degree of the divisor of poles $s$ depends on the order of $G$. 
Recalling that the rotation groups of a pyramid, bypyramid and platonic solids, 
the proof is done. 

As a second goal, 
recall that $\omega \in \racs{s}$ is isochronous when all their residues are purely imaginary. 
We study the subfamily $\RI{s} \subset \racs{s}$ of isochronous 1--forms. 
Our result is below.
\begin{rcorollary}[\ref{teorema-RI-variedad}]
 \textit{The subfamily $\RI{s}$ is a $(3s-1)$--dimensional real analytic submanifold of $\racs{s}$.}
\end{rcorollary}
For the proof, 
we use the complex atlas by residues--poles. 
Proposition \ref{realizacion-grupos-finitos} is fulfilled for suitable $\omega \in \RI{s}$. 
The residues are naturally a set of $\PSL$--invariant functions. 
They can be used in order to describe a realization of 1--forms; 
see \cite{isocrono}. 
The $\PSL$--action on $\RI{s}$ is well--defined. 
Hence, 
the quotients $\racs{s}/\PSL$ and $\RI{s}/\PSL$ admit a complex and a real stratification by orbit types, repectively. 
For $s \geq 4$, in order to get a complete set of $\PSL$--invariant functions we enlarge the set of residues by adding the cross--ratio of poles and explicitly recognize realizations for the quotients $\racs{s}/\PSL$ and $\RI{s}/\PSL$; 
see respectively Proposition \ref{quotient} and Corollary \ref{quotient2}.\\
In Section \ref{sec-surfaces}, for $\omega \in \racs{s}$, resp. $\omega \in \RI{s}$, we obtain a realization for the quotient of the associated flat surfaces $S_{\omega}$ up to isometries $\mathfrak{M}(-s)$, resp. $\mathcal{RI}\mathfrak{M}(-s)$, extending naturally the $\PSL$--action.\\
 
Our results have applications in dynamical systems and the geometry of flat surfaces since there is a one--to--one \textit{correspondence} between rational 1--forms $\omega = (Q(z)/P(z)) dz$, 
oriented rational quadratic differentials $\omega \otimes \omega = (Q^2(z)/ P^2(z))dz^2$, 
rational complex vector fields $X_{\omega} = P(z)/Q(z) \ \partial/\partial z$, 
pairs of singular real analytic vector fields $(\Re{X_{\omega}},$ $\Im{X_{\omega}})$ on $\CC \setminus \{ Q = 0\}$, 
and singular flat surfaces $S_{\omega}=(\CC , g_{\omega})$ provided with two real singular geodesic foliations. 
The metrics $g_{\omega}$ are the associated to the quadratic differentials $\omega \otimes \omega$, 
and the foliations come from the horizontal and vertical trajectories. 
This correspondence is used in many works, \textit{e.g.} \cite{kerckhoff,jesus1, alvaro1}.
For $\omega \in \racs{s}$, 
its associated quadratic differential $\omega \otimes \omega$ has poles of multiplicitie 2 and they were studied by K. Strebel in \cite{strebelB, strebel}.
\\

Historically, C. Huygens \cite[p.~72]{whittaker} gave formulas for the period of isochronous centers in the model of a simple pendulum as differential form $(i\lambda/z)dz$.
Trying to reach a contemporary point of view, we recall the following statements.\\
Quadratic differentials with closed trajectories were first considered by O. Teich-m\"uller in his \textit{``Habilitationsschrift"} \cite{teichmuller}. 
On $M_g$, 
K. Strebel \cite{strebelB, strebel} proved that certain quadratic differentials with closed regular horizontal trajectories realize the extremal metric problem introduced by J. A. Jenkins \cite{jenkins}. \\
Another interesting facet of isochronous 1--forms $\omega \in \RI{s}$, comes from dynamical systems. 
The phase portrait of the associated vector field $\Re{X_{\omega}}$ is a union of isochronous centers or annulus. 
In other words, any pair of trajectories of $\Re{X_{\omega}}$ that share a center basin have the same period. 
Looking at isochronous centers on $\R^2$; 
P. Marde{\v{s}}i\'c, C. Rousseau and B. Toni \cite{mardesic} studied the linearization problem, 
and L. Gavrilov \cite{gavrilov} considered the appearance of isochronous centers in polynomial of Hamiltonian systems on $\C^2$ and its relation to the famous Jacobian conjecture. 
A constructive result for isochronous vector fields on $\C$, by using the residues, is provided by J. Muci\~no-Raymundo in \cite[\S~8]{jesus1}. 
A topological and analytic classification of complex polynomials vector fields on $\C$ with only isochronous centers was performed by M. E. Fr\'ias-Armenta and J. Muci\~no--Raymundo in \cite{isocrono}.

\section{Rational 1--forms with simple poles}
\label{S-families}

\subsection{Stratification}

We define a stratification on the set $\racs{s}$ of rational 1--forms on the Riemann sphere \CC, 
having exactly $-s \geq -2$ simple poles. 
First, recall that the infinite dimensional vector space of rational 1--forms admits a stratification fixing the multiplicities of the zeros $\{k_1, \ldots, k_m\}$ and poles $\{-s_1, \ldots, -s_n\}$, 
where $k_j, s_{\iota} \in \N$. 
The stratum of rational 1--forms with these multiplicities is denoted by $\Omega^1\{k_1, \ldots, k_m ; -s_1, \ldots, -s_n \}$, 
they are connected $(m+n+1)$--dimensional complex manifolds in the vector space of rational 1--forms; 
our notation is similar to \cite{jesus1, moguel2}. 
On the other hand, 
for $g \geq 2$ the stratum are not necessarily connected. 
M. Kontsevich and A. Zorich describe the connected components of holomorphic 1--forms in each stratum on $M_g$; 
see \cite{kontsevich}.\\
Our framework allows us to study 
\begin{equation}\label{omega-variedad}
\racs{s} = \bigsqcup \Omega^1\{ k_1, \ldots, k_m ; \underbrace{-1,\ldots, -1}_{\text{$s$}} \},
\end{equation}
where the union takes all the multiplicities $\{k_1, \ldots, k_m; -1, \ldots, -1 \}$ such that $\{k_1, \ldots , k_m \}$ is an integer partition of $s-2$, 
\textit{i.e.} the sum $k_1 + \ldots + k_m = s - 2$. 
The expression $(-1^s)$ in \eqref{omega-variedad} is motivated by the ``exponential" notation for multiple  poles of the same degree in the stratification by multiplicities of rational 1--forms; see \cite{kontsevich, corentin}. 

\subsection{Polynomials}\label{polynomials}

Let us recall the existence of two natural complex atlases for polynomials $\C[z]_{=s}$ with degree $s$. For complex manifolds, we use notation as in \cite[Ch.~IV]{fritzsche}. 
First, given the natural homeomorphism
$$
\begin{array}{rcl}
f_1: \C[z]_{=s} & \longrightarrow & \C^{s+1} \setminus \{ b_s = 0\} \\
b_sz^s + \ldots + b_0 & \longmapsto & (b_s, \ldots, b_0),
\end{array}
$$
we obtain that $(\C[z]_{=s}, f_1)$ is an $(s+1)$--dimensional complex coordinate system in $\C[z]_{=s}$. 
Clearly, the subset of polynomials with degree $s$ and simple roots $\C[z]_{=s} \setminus \mathcal{D}(P, P')$ is open, 
where $\mathcal{D}(P, P')$ denotes the discriminant of $P$. \\
Second, the action of the symmetric group $\Sim{s}$ of $s$ elements on $\CC^{s} \setminus \Delta := \{(p_1, \ldots, p_s) \in \CC^{s} \ | \ p_{\iota} \not= p_{\kappa}, \ \text{for all }\iota \not= \kappa \}$ is properly discontinous. 
In fact, the quotient $(\CC^s \setminus \Delta) / \Sim{s}$ is an $s$--dimensional complex manifold. 
Moreover, 
$\left(\C[z]_{=s}\setminus \mathcal{D}(P,P'), \nu_s \circ f_2 \right)$ is an $(s+1)$--dimensional complex coordinate system in $\C[z]_{=s} \setminus \mathcal{D}(P,P')$, 
where the map
$$
\begin{array}{rcl}
f_2: \C[z]_{=s} \setminus \mathcal{D}(P,P') & \longrightarrow  & \C^* \times \left(\frac{\CC^{s} \setminus \Delta }{\Sim{s}} \right) \\ 
b_s\prod_{\iota=1}^s (z - p_{\iota}) & \longmapsto & (b_s, \{p_1, \ldots, p_s \})
\end{array}
$$
is a natural bijection and
\begin{equation}\label{viete-map}
\begin{array}{rcl}
\nu_s : \C^* \times \left( \frac{\CC^s \setminus \Delta}{\Sim{s}} \right) & \longrightarrow & \C^{s+1} \setminus \{ b_s = 0 \} \\
(b_s, \{p_1, \ldots, p_s\}) & \longmapsto & \left(b_s, -b_s\left(\sum_{\iota=1}^s p_{\iota}\right), \ldots, (-1)^sb_s\prod_{\iota=1}^s p_{\iota} \right),
\end{array}
\end{equation} 
is called the \textit{Vi\`ete} map. \\
Finally, 
it is easy to prove that the coordinate systems $(\C[z]_{=s}\setminus \mathcal{D}(P,P'), f_1)$ and $(\C[z]_{=s} \setminus \mathcal{D}(P,P'), \nu_s \circ f_2)$ are hollomophically compatible; 
see \cite{katz}. 
A remarkable fact is that some properties of polynomials are easy to see in one coordinate system but others are kept hidden.\\

\subsection{Complex atlases on $\racs{s}$}

Similarly as in Section \ref{polynomials}, 
three equivalents $(2s-1)$--dimensional complex atlases on $\racs{s}$ will be constructed.

\begin{description}[leftmargin=0cm]
\item[1) \textbf{Coefficients.}] We consider the Zariski open subset of $\CP^{2s-1}$,  
$$
\Omega_{coef}^{1}(-1^s) := \left\{ [a_{s-2} : \ldots : a_0 : b_s : \ldots : b_0 ] \in \CP^{2s-1} \ \ \left| \ \  \begin{array}{rcl}
														\mathcal{D}(P,Q) &\not=& 0, \\
														\mathcal{D}(P,P') &\not=& 0	
														\end{array} \right. \right\},
$$
where $Q(z) = a_{s-2}z^{s-2} + \ldots + a_0$, $P(z) = b_sz^{s} + \ldots + b_0$ and $\mathcal{D}(P,Q)$ denotes the resultant of the polynomials $P$ and $Q$. 
If $\{ (U^{coef}_j, \varphi^{coef}_j)\}$ denotes the complex atlas on $\Omega^1_{coef}(-1^s)$, 
then $\mathfrak{A}_{coef} := \{ (f_{coef}^{-1}(U^{coef}_j), \varphi^{coef}_j \circ f_{coef} )\}$ is a complex atlas on $\racs{s}$, 
where
$$
\begin{array}{rcl}
f_{coef}: \racs{s} & \longrightarrow & \Omega^1_{coef}(-1^s) \\
\displaystyle \omega = \frac{a_{s-2}z^{s-2} + \ldots + a_0}{b_sz^s + \ldots + b_0}dz & \longmapsto & [a_{s-2} : \ldots : a_0 : b_s : \ldots : b_0],
\end{array}
$$
is a natural bijection map. 
\item[2) \textbf{Zeros--poles.}] We consider
$$
M_s := \left.\left\{\left\{c_1, \ldots, c_{s-2}, p_1, \ldots, p_s\right\} \in \frac{\CC^{s-2}}{\Sim{s-2}} \times \frac{\CC^s \setminus \Delta}{\Sim{s}} \ \right| c_j \not= p_{\iota} \ \right\}.
$$
Recalling that there exists a biholomorphism between $\CC^s/\Sim{s}$ and $\CP^s$, $M_s$ is a $(2s-2)$--dimensional open and dense complex submanifold of $\CP^{s-2}\times \CP^s$.  
Consider the transition functions for a nontrivial principal $\C^{*}$--bundle over $M_s$ as follows. 
Let $u= \left\{ c_1, \ldots, c_{s-2}, p_1, \ldots, p_s \right\}	\in M_s$ and set
$$
	Q_u(z) := \left\{ \begin{array}{ll}
			(z-c_1)\cdots (z-c_{s-2}) & \text{where $c_j\not= \infty$ for all $j$,} \\
			(z-c_1)\cdots(z - c_{j-1})(z- c_{j+1})\cdots(z-c_{s-2}) & \text{if $c_j=\infty$,}
		       \end{array}\right.
$$
$$
	P_u(z) := \left\{ \begin{array}{ll}
			(z-p_1)\cdots (z-p_{s}) & \text{where $p_{\iota}\not= \infty$ for all $\iota$,} \\
			(z-p_1)\cdots(z - p_{\iota-1})(z- p_{\iota+1})\cdots(z-p_{s}) & \text{if $p_{\iota}=\infty$.}
		       \end{array}\right.
$$
For $\alpha \in  I := \{1, 2, \ldots, 2s-1 \}$, 
we define $\Omega_{zp}^{1}(-1^s)$ as the total space of the principal $\C^*$--bundle over $M_s$ with the transition functions 
\begin{equation}{\label{funcion-transicion}}
\begin{array}{rcl}
		\phi_{\alpha \beta} : \mathcal{U}_{\alpha} \cap \mathcal{U}_{\beta} &\longrightarrow & \C^{*} \\
		u                                                                                                                                   & \longmapsto &\displaystyle \frac{Q_u(\alpha)}{P_u(\alpha)}\left(\frac{P_u(\beta)}{Q_u(\beta)}\right),
\end{array}
\end{equation} 
where $\mathcal{U}_\alpha := \left\{ u \in M_s \  | \  Q_u(\alpha) \not= 0 \text{ and } P_u(\alpha)  \not= 0 \right\}$. 
If $\{(U^{zp}_j, \varphi^{zp}_j)\}$ denotes the complex atlas on $\Omega^1_{zp}(-1^s)$, 
then $\mathfrak{A}_{zp} = \{ (f_{zp}^{-1}(U^{zp}_j), \varphi^{zp}_j \circ f_{zp} ) \}$ is a complex atlas on $\racs{s}$ where 
$$
\begin{array}{rcl}
f_{zp}: \racs{s} & \longrightarrow & \Omega^1_{zp}(-1^s) \\
\displaystyle \omega = \lambda\frac{(z-c_1)\cdots(z-c_{s-2})}{(z-p_1)\cdots(z-p_s)} dz & \longmapsto & [\{c_1, \ldots c_{s-2}, p_1, \ldots,  p_s\}, \lambda],
\end{array}
$$
is a bijection map.
\item[3) \textbf{Residues--poles.}] Let $H_s := \{(r_1, \ldots, r_s) \in (\C^*)^s \ | \ r_1 + \ldots + r_s = 0 \}$ and recall that $\CC^{s} \setminus \Delta = \{(p_1, \ldots, p_s) \in \CC^{s} \ | \ p_{\iota} \not= p_{\kappa}, \ \text{for all }\iota \not= \kappa \}$. 
We consider a diagonal action of the symmetric group $\Sim{s}$ of $s$ elements
	\begin{equation}\label{accion-residuos-polos}
	 \begin{array}{lll}
	\Sim{s} \times (H_s \times (\CC^s\setminus \Delta)) &\longrightarrow &H_s \times (\CC^s\setminus \Delta) \\
	(\sigma, (r_1, \ldots, r_s, p_1, \ldots, p_s))        &\longmapsto      & (r_{\sigma(1)}, \ldots, r_{\sigma(s)}, p_{\sigma(1)}, \ldots, p_{\sigma(s)}).
	\end{array}
	\end{equation}
Clearly, 
the action above is properly discontinuous and the quotient
$$
\Omega_{rp}^{1}(-1^s) := \frac{H_s \times (\CC^{s} \setminus \Delta)}{\Sim{s}}.
$$
is a $(2s-1)$--dimensional open complex manifold. 
We denote the equivalence class under the action \eqref{accion-residuos-polos} as $\rp{r_1, \ldots, r_s; p_1, \ldots, p_s}$. 
Geometrically, 
an element in $\Omega_{rp}^{1}(-1^s)$ is a configuration\footnote{We convene that a configuration is an unordered set of points different between them.} 
of $s$ points $\{p_{\iota}\}$ in the Riemann sphere with weights $\{r_{\iota}\} \subset \C^*$ which satisfy the residue theorem.
If $\{(U^{rp}_j, \varphi^{rp}_j)\}$ denotes the complex atlas on $\Omega^1_{rp}(-1^s)$, 
then $\mathfrak{A}_{rp} := \{ (f_{rp}^{-1}(U^{rp}_j), \varphi^{zp}_j \circ f_{zp} ) \}$ is a complex atlas on $\racs{s}$ where 
$$
\begin{array}{rcl}
f_{rp}: \racs{s} & \longrightarrow & \Omega^1_{rp}(-1^s) \\
\displaystyle \omega = \sum_{\iota = 1}^s \frac{r_{\iota}}{z-p_{\iota}} dz & \longmapsto &  \rp{r_1,\ldots,r_s;p_1, \ldots ,p_s}, 
\end{array}
$$
is a bijection map. 
For a 1--form $\omega$ with the pole $p_{\kappa} = \infty$, 
the term $r_{\kappa}/(z - p_{\kappa})$ is omitted in the sum above. 
Obviously, 
the residue theorem is the unique obstruction to realize $\omega \in \racs{s}$. 
Moreover, 
$$
2s-1 = dim_{\C} (\racs{s}) \geq dim_{\C} (\Omega^1 \{ k_1, \ldots, k_m ; \underbrace{-1,\ldots, -1}_{\text{$s$}} \}) = m+s+1.
$$
\end{description}
Recall that the complex atlases above are valid only for rational 1--forms with simple poles. 
The study of rational 1--forms with poles of multiplicitie grater or equal than 2, 
will be consider in a future work. 
Our main result is as follows. 

\begin{theorem}\label{teorema-equivalencia-parametros}
	The complex manifolds $\Omega_{coef}^{1}(-1^s)$, $\Omega_{zp}^{1}(-1^s)$ and $\Omega_{rp}^{1}(-1^s)$ are biholomorphic.
\end{theorem}

\begin{proof}
We construct explicitly two biholomorphisms from $\Omega_{coef}^1(-1^s)$ to $\Omega_{rp}^1(-1^s)$ and $\Omega_{zp}^1(-1^s)$ to $\Omega_{coef}^1(-1^s)$.\\ 
First, 
we show that $\Omega^{1}_{rp}(-1^s)$ and $\Omega^{1}_{coef}(-1^s)$ are biholomorphic. 
For $\rp{P} = \left\langle r_1, \ldots,\right.$ $\left. r_s; p_1, \ldots, p_s \right\rangle \in \Omega_{rp}^{1}(-1^s)$, 
consider $C_s:\Omega_{rp}^{1}(-1^s)$ $\longrightarrow  \Omega_{coef}^{1}(-1^s)$ such that
$$
C_s\rp{P} := \left\{\begin{array}{lll}
										\displaystyle \left[-\sum_{j= 1}^s r_j \sum_{\iota \not= j} p_{\iota} :\sum_{j=1}^{s} r_j \sum_{\iota, \kappa \not= j} p_{\iota}p_{\kappa}:\ldots: \right.&& \\
\displaystyle \left. (-1)^{s-1}\sum_{j=1}^{s}r_j \prod_{\iota \not= j}p_{\iota}: \nu_{s}(1, \left\{p_1,\ldots,p_s\right\} )\right] &\text{for }p_{\iota} \in \C, & \\
										&& \\
										\displaystyle \left[\sum_{j \not= \kappa }^s r_j :-\sum_{j \not= \kappa }^{s} r_j \sum_{\iota \not= j} p_{\iota}:\ldots:\right.& &\\
\displaystyle \left. (-1)^{s-2} \sum_{j \not= \kappa }^{s} r_j \prod_{\iota \not= j} p_{\iota}:0: \nu_{s-1}(1,\left\{p_1,\ldots,\widehat{p}_{\kappa}, \ldots, p_s\right\}) \right] &\text{for }p_{\kappa} = \infty, & 
										\end{array}\right.	
$$
where $\nu_{s}$ is the \textit{Vi\`ete} map in \eqref{viete-map}. 
The hat over the pole $p_{\kappa}$ indicates that it is omitted. 
A direct computation prove that the map $C_s$ is a bijective map. 
If $p_j \not= \infty$ for all $j=1,\ldots s$, 
then the Jacobian matrix is
$$
D{C}_s(r_1,\ldots,r_s,p_1,\ldots,p_s) = \left(\begin{array}{c|c} 
								A & * \\  \hline
								0 &  D{\nu}^{o}_{s}\{p_1,\ldots,p_s\} 
							     \end{array}\right),
$$
where
\begin{equation}\label{matriz-vieta}
							A = \left(\begin{array}{cccc}
							\displaystyle \sum_{j\not=1} p_j  &\displaystyle  \sum_{j\not=2} p_j & \ldots &\displaystyle  \sum_{j\not=s} p_j \\
							\displaystyle -\sum_{j, \iota \not=1} p_jp_{\iota}  &\displaystyle -\sum_{j, \iota \not=2} p_jp_{\iota} & \ldots &\displaystyle -\sum_{j, \iota \not=s} p_jp_{\iota} \\
							\vdots & \vdots & & \vdots \\
							\displaystyle (-1)^{s}\prod_{j\not=1} p_j  &\displaystyle  (-1)^{s}\prod_{j\not=2} p_j & \ldots &\displaystyle  (-1)^{s}\prod_{j\not=s} p_j \\
					    \end{array}\right),
\end{equation}
and $\nu^o_s$ denotes the \textit{Vi\`ete} map $\nu_s$ by removing the first coordinate. 
The rows of $D{C}_s$ are linear independent, 
and the map $C_s$ is a biholomorphism. 
The case $p_j = \infty$ is analogous. 
We are done, $\Omega^{1}_{coef}(-1^s)$ is biholomorphic to $\Omega^{1}_{rp}(-1^s)$.\\

Secondly, we prove that $\Omega_{zp}^{1}(-1^s)$ and $\Omega_{coef}^{1}(-1^s)$ are biholomorphic. 
For $u = \left\{ c_1,\ldots,\right.$ $\left.c_{s-2},p_1,\ldots,p_s \right\} \in \mathcal{U}_{\alpha} \subset M_s$, 
we consider the map
\begin{eqnarray*}
\mathcal{F}_s: \Omega_{zp}^{1}(-1^s) & \longrightarrow & \Omega_{coef}^{1}(-1^s) \\
           \left[ u, \lambda \right] & \longmapsto &  \left[\lambda \frac{P_u(\alpha)}{Q_{u}(\alpha)}\nu_{s-2}(1, \left\{c_1, \ldots, c_{s-2}\right\}) : \nu_s(1, \left\{p_1, \ldots, p_s \right\}) \right].
\end{eqnarray*} 
For all $\alpha, \beta \in I$, the transition functions $\{ \phi_{\beta \alpha} \}$ make that the diagram below conmutes.
$$
		\xymatrix{ (\mathcal{U}_{\alpha} \cap \mathcal{U}_{\beta}) \times \C^{*} \ar@{->}^{\phi_{\beta \alpha}}[rr] \ar@{->}_{\mathcal{F}_s}[rd] & & (\mathcal{U}_{\alpha} \cap \mathcal{U}_{\beta}) \times \C^{*} \ar@{->}^{\mathcal{F}_s}[ld] \\
                                                               &\Omega_{coef}^{1}(-s), &
				}
$$ 
In fact, the map $\mathcal{F}_s$ is a biholomorphism.  
\end{proof}

From now, 
we only use the complex atlas by residues--poles $\mathfrak{A}_{rp}$ on $\racs{s}$. 
However, 
by Theorem \ref{teorema-equivalencia-parametros} the results in this paper are valid independently of the complex atlas.

\begin{definition}
A rational 1--form $\omega \in \racs{s}$ is \emph{isochronous} when all their residues are purely imaginary. 
The family of rational isochronous 1--forms is denoted by
$$
 \RI{s} :=\left\{ \omega \in \racs{s}  \ \left| \  \omega \text{ is isochronous} \right. \right\}.
$$
\end{definition}

\begin{corollary}\label{teorema-RI-variedad}
The subfamily $\RI{s}$ is a $(3s-1)$--dimensional real analytic submanifold of $\racs{s}$.
\end{corollary}

\begin{proof}
The result follows using the complex atlas by residues--poles $\mathfrak{A}_{rp}$ and the $(s-1)$--dimensional real analytic submanifold $\Im{H_s} := \{(ir_1, \ldots , ir_s) \in H_s \ | \  r_{\iota} \in \R^*\}$ of $H_s$. 
\end{proof}

\section{Classification of isotropy groups}
\label{Sec:action}

\subsection{The $\PSL$-action}

In this section, 
we prove that the natural holomorphic $\PSL$--action on $\racs{s}$, 
defined as
\begin{equation}\label{PSL-action}
	\begin{array}{rcl}
		\mathcal{A}_s: \PSL \times \racs{s} &\longrightarrow& \racs{s} \\
		(T, \omega)                       &\longmapsto& T_*\omega,
	\end{array}
\end{equation}
is proper for $s\geq 3$. 
\begin{remark}
\begin{upshape}
Using the complex atlas by residues--poles $\mathfrak{A}_{rp}$, 
the expression for the action is  
$$
\mathcal{A}_s(T, \rp{r_1, \ldots r_s; p_1, \ldots, p_s}) = \rp{r_1, \ldots, r_s ; T(p_1), \ldots, T(p_s)}.
$$
The residues are a set of $\PSL$--invariant functions under the above action.
\end{upshape}
\end{remark}
The class of an $\omega$ is denoted by $\rp{\rp{\omega}} \in \racs{s}/\PSL$. 
Recall the definition of proper action as in \cite[p.~53]{diustermaat}, 
we have the next result. 
\begin{lemma}\label{proper-action}
 For $s \geq 3$, the holomorphic (resp. real analytic) $\PSL$--action $\mathcal{A}_s$ on $\racs{s}$ (resp. on $\RI{s}$) is proper.
\end{lemma}

\begin{proof}
We will show that the map $\tilde{\mathcal{A}}_s : \PSL \times \racs{s}  \longrightarrow  \racs{s} \times \racs{s}$, 
defined as $\tilde{\mathcal{A}}_s(T, \omega) := (T_*\omega, \omega)$, 
is closed and the preimage for all points is a compact set. 
Applying Thm. 1 in \cite[Sec.~\textsection~10.2~p.~101]{bourbaki}, 
the action $\mathcal{A}_s$ is proper.

First, we want to prove that the map $\tilde{\mathcal{A}}_s$ is closed. 
Consider a closed subset $C \subset \PSL \times \racs{s}$ and a convergent sequence $\{(\eta_m, \omega_m )\} \subset \tilde{\mathcal{A}}_s(C)$ with a limit point $(\eta, \omega) \in \racs{s} \times \racs{s}$. 
Since $(\eta_m, \omega_m) \in \tilde{\mathcal{A}}_s(C)$, the sets of resdiues for $\omega_m$ and $\eta_m$ coincide.  
Explicity, for each $m$ we choose $\eta_m =\rp{r'_{m1}, \ldots, r'_{ms}; q_{m1} ,\ldots, q_{ms}}$ and $\omega_m = \rp{r_{m1}, \ldots, r_{ms}; p_{m1} ,\ldots, p_{ms}}$, 
without loss of generality $r'_{m\iota}=r_{m\iota}$; 
here our assertions will work for all $\iota =1, \ldots, s$. 
If $s \geq 3$, then there exists a unique $T_m \in \PSL$ such that $T(p_{m\iota})=q_{m\iota}$. 
Since $C$ is closed and $(\eta, \omega)$ is the limit point of the sequence $\{(\eta_m, \omega_m)\}$, 
say $\eta = \rp{r'_1, \ldots, r'_s; q_1 ,\ldots, q_s}$ and $\omega = \rp{r_1, \ldots, r_s; p_1 ,\ldots, p_s}$; 
it follows that there is a unique limit transformation $T \in \PSL$ with $T(p_{\iota})=q_{\iota}$ and $r'_{\iota} = r_{\iota}$, 
thus, the sequence $\{T_{m}\}$ converges to $T$. 
Therefore $(\eta,\omega) \in \tilde{\mathcal{A}}_s(C)$, and the map $\tilde{\mathcal{A}}_s$ is closed.

Secondly, 
we prove that $\mathcal{A}^{-1}_s(\eta,\omega)$ is a compact set. 
Since there are at most $s!$ permutations of the configuration of poles $\{p_{\iota}\}$ to $\{q_{\iota}\}$, 
a configuration of poles with residues $\rp{r_1, \ldots, r_s; p_1, \ldots, p_s}$ has at least two residues satisfying $r_{\iota} \not= r_j$; 
hence there are at most $(s-1)!$ admissible permutations of $\rp{r_1, \ldots, r_s; p_1, \ldots, p_s}$ to $\rp{r'_1,\ldots, r'_s; q_1, \ldots, q_s}$. 
In fact, $\tilde{\mathcal{A}}_s^{-1}(\eta,\omega)$ is a finite set, 
hence compact in $\PSL \times \racs{s}$. 
\end{proof}

\subsection{Nontrivial isotropy groups}

For $\omega \in \racs{s}$, 
we denote by

\centerline{$\PSL_{\omega} := \{ T \in \PSL \ | \ T_*\omega = \omega \}$, } 

\noindent its isotropy group. 

\noindent A direct computation prove that $\tilde{\mathcal{A}}_s^{-1}(\omega,\omega) = \PSL_{\omega} \times \{\omega \}$ is a finite set when $s \geq 3$; 
see prove of Lemma \ref{proper-action}. 
In fact, 
every $\omega \in \racs{s}$ has finite isotropy group. 
A well--known result of F. Klein \cite[p.~126]{klein20} classifies the finite subgroups of $\PSL$; 
for a modern reference see \cite[Sec.~2.13]{singerman}. 
These finite subgroups are cyclic $\Z_n$, 
dihedral $D_n$ and the rotation groups $G(S)$ of platonic solids $S$; 
$A_4$ for tetrahedron, 
$\Sim{4}$ for octahedron and cube, 
and $A_5$ for dodecahedron and icosahedron. 
A natural question is which finite subgroups of $\PSL$ are realizable as isotropy groups of $\omega \in \RI{s}$? 
The answer is as follows.
\begin{proposition}\label{realizacion-grupos-finitos}
 Every finite subgroup $G < \PSL$ appears as the isotropy group of suitable $\omega \in \RI{s}$.
\end{proposition}

\begin{proof}
Fixing $G < \PSL$ finite subgroup, 
we construct explicitly a rational 1--form $\omega \in \RI{s}$ such that $\PSL_{\omega} \cong G$. 
Consider $\zeta_1, \ldots, \zeta_n$ the $nth$ roots of unity; $n \geq 2$.\\ 
\textit{Case $G=\Z_n$.} Recall that the rotation group of  a pyramid with polygonal base and triangular faces is $\Z_n$, 
where $n$ is the number of sides on the base. 
In particular, 
the set $\{\zeta_1, \ldots, \zeta_n, 0 \}$ in the Riemann sphere are the vertices of a pyramid as above. 
In fact, 
if 
$$
\omega = \rp{\underbrace{i, \ldots, i}_{n}, -ni; \zeta_1, \ldots, \zeta_n, 0} \in \RI{(n+1)},
$$ 
then its isotropy group is $\PSL_{\omega} \cong \Z_n$. \\ 
\textit{Case $G=D_n$.} A bypyramid is a polihedron defined by two pyramids glued together by their basis. 
If all their faces are isosceles triangles, 
then its rotation group is $D_n$ where $n$ are the number of sides in the base for both pyramids. 
In particular, 
the set $\{\zeta_1, \ldots, \zeta_n, 0, \infty \}$ in the Riemann sphere are the vertices of a bypyramid as above. 
In fact, 
if
\vspace{-0.2cm}
$$
\omega=\rp{\underbrace{i,\ldots, i}_{\text{$n$}},-\frac{n}{2}i, -\frac{n}{2}i; \zeta_1,\ldots,\zeta_n, 0, \infty} \in \RI{(n+2)},
$$
then its isotropy group is $\PSL_{\omega} \cong D_n$.\\
\textit{Case $G=G(S)$.} We consider the union of the vertices of a platonic solid $S$ and its dual $S^{*}$ in the Riemann sphere. The suitable 1--form $\omega$ has poles in both sets of vertices. 
The choice of the residues is as follows, 
residue $i$ at the vertices of $S$ and $-ki$ at the vertices of $S^{*}$, where
$$
k := 	\left\{\begin{array}{ll}
		1 & \text{for $S$ the tetrahedron}, \\
     	\vspace{-0.4cm} & \\
		{4}/{3} \hspace{1cm}& \text{for $S$ the cube}, \\
     	\vspace{-0.4cm} & \\
     {3}/{4} & \text{for $S$ the octahedron}, \\
		\vspace{-0.4cm} & \\
		{3}/{5} & \text{for $S$ the dodecahedron}, \\
     	\vspace{-0.4cm} & \\
     {5}/{3} & \text{for $S$ the icosahedron}, \\ 
	\end{array}\right. 
$$
whence the isotropy group $\PSL_{\omega} \cong G(S)$. 
Concrete examples are provided below. 
\end{proof}

\begin{example}
\begin{upshape}
1. For $\epsilon_1, \epsilon_2, \epsilon_3$ roots of $z^{3} + 1 = 0$, the isotropy group of
$$
\begin{array}{ll}
\displaystyle \omega=\rp{\underbrace{i,\ldots, i}_{\text{4}},\underbrace{-i,\ldots, -i}_{\text{4}}; \frac{\sqrt{2}}{2}\zeta_1,  \frac{\sqrt{2}}{2}\zeta_2,  \frac{\sqrt{2}}{2}\zeta_3, \infty, {\sqrt{2}}{\epsilon_1},  {\sqrt{2}}{\epsilon_2}, {\sqrt{2}}{\epsilon_3}, 0 } &\\
&\hspace{-1.2cm} \in \RI{8}
\end{array}
$$
is isomorphic to the rotation $A_4$ group of a tetrahedron.

\noindent 2. For $\epsilon_{1}, \epsilon_2, \epsilon_3, \epsilon_4$ roots of $z^4+1=0$ and $\lambda={(\sqrt{6} - \sqrt{2})}/{2}$, the isotropy group of 
$$
\begin{array}{ll}
 \omega= \left\langle \underbrace{i,\ldots, i}_{\text{8}}, \underbrace{-\frac{4}{3}i, \ldots, -\frac{4}{3}i}_{\text{6}}; \right. & \\
&\displaystyle \hspace{-2.8cm} \left. \lambda, -\lambda, i\lambda, -i\lambda, \frac{1}{\lambda}, -\frac{1}{\lambda}, \frac{i}{\lambda}, -\frac{i}{\lambda}, \epsilon_1, \epsilon_2, \epsilon_3, \epsilon_4, 0, \infty \right\rangle \in \RI{14}
\end{array}
$$
is isomorphic to the rotation group $\Sim{4}$ of a cube (octahedron).
\end{upshape}
\end{example}

 Obviously, the degree $s$ depends on the order of $G$ and the $1$--forms in the above proposition are isochronous. 
A. Solynin \cite{solynin} constructs quadratic differentials on compact Riemann surface $\mathcal{R}$ associated with a weight graph embedded in $\mathcal{R}$. 
He explicitly gives quadratic differentials, with zeros in the vertices of a platonic solid and poles in the center of each face. 
They are different from our 1--forms. 
A. Alvarez--Parrila, M. E. Fr\'ias--Armenta and C. Yee--Romero \cite{alvaro3} classify the isotropy groups of rational 1--forms on the Riemann sphere using the complex atlas by zeros--poles. 

For $2 \leq s \leq 11$, we classify the isotropy groups $\PSL_{\omega}$. These results will help in Section \ref{S-examplesquotient}.

\begin{example}\label{ejemplo-S2}
\begin{upshape}
1. For all $\omega= \rp{r, -r; p_1, p_2} \in \racs{2}$, 
the isotropy group is $\PSL_{\omega} \cong \C^* \cong \{T(z) = az \}$.

\noindent 2. For $\omega \in \racs{3}$, the isotropy group is $\PSL_{\omega}\cong \Z_2$ if and only if $\omega= \rp{r_1, r_1, r_2; p_1, p_2, p_3}$, \textit{i.e.} two residues are equal.
\end{upshape}
\end{example}

\begin{lemma}\label{isotropia-residuos-4}
	Consider $\omega=\rp{r_1,r_2 ,r_3, r_4; p_1, p_2, p_3, p_4} \in \racs{4}$.\\
\noindent 1. If $\omega$ has exactly two equal residues and the cross--ratio\footnote{The cross--ratio is defined as $(p_1, p_2, p_3, p_4) := \frac{(p_4-p_1)(p_3-p_2)}{(p_4-p_2)(p_3-p_1)}.$} 
\begin{itemize}
\item[] $(p_1, p_2, p_3, p_4) \in \{-1, \frac{1}{2}, 2\}$, then $\PSL_{\omega} \cong \Z_2$.
\end{itemize}
\noindent 2. If $\omega$ has two pairs of equal residues and
			\begin{itemize}
				\item[] $(p_1, p_2, p_3, p_4) \not\in \{-1, \frac{1}{2}, 2\}$, then $\PSL_{\omega} \cong \Z_2$,
				\item[] $(p_1, p_2, p_3, p_4) \in \{-1, \frac{1}{2}, 2\}$, then $\PSL_{\omega} \cong \Z_2 \times \Z_2 \cong D_2$.
			\end{itemize}
\noindent 3. If $\omega$ has three equal residues and 
\begin{itemize}
\item[] $(p_1, p_2, p_3, p_4) \in \left\{ (1 \pm i \sqrt{3})/ 2 \right\}$, then $\PSL_{\omega} \cong \Z_3$.
\end{itemize}
\noindent 4. For any other case the isotropy group $\PSL_{\omega}$ is trivial.
\end{lemma}

\begin{proof}
\noindent Case 1. 
Consider $\omega = \rp{r_1, r_2, r_3, r_4; p_1, p_2, p_3, p_4} \in \racs{4}$ with $r_1=r_2$, $r_3\not=r_4$ and $(p_1, p_2, p_3, p_4)=-1$. 
We can verify that $(p_1, p_2, p_3, p_4)=(p_2, p_1, p_3, p_4)$, 
and there is a nontrivial $T \in \PSL$ such that $T_{*}\omega = \left\langle r_1, r_2, r_3, \right.$ $\left. r_4; p_2, p_1, p_3, p_4 \right\rangle = \omega$. 
In fact, 
$T \in \PSL_{\omega}$. 
Since $r_3\not= r_4$, 
there are no more elements in the isotropy; 
therefore $\PSL_{\omega} \cong \Z_2$. 
The result is analogous for
$$
\begin{array}{cccccc}
r_1 \not=r_2, & r_3=r_4, & \lambda=-1, &  r_1 = r_4, & r_2\not=r_3, & \lambda = 2, \\
r_1 = r_3, & r_2\not=r_4, & \lambda = {1}/{2}, &  r_1 \not= r_4, & r_2=r_3, & \lambda = 2,\\
r_1 \not= r_3, & r_2=r_4, & \lambda = {1}/{2}.& & & \\
\end{array}
$$
We leave the reader to perform the other cases. 
\end{proof}

Obviously, for $s \geq 5$ the specific conditions to determine the isotropy groups are more complicated.

\begin{example}{\label{isotropia-residuos-5}}
\begin{upshape}
	For $\omega \in \racs{5}$, the nontrivial isotropy groups $\PSL_{\omega}$ are isomorphic to $\Z_{2}, \Z_{3}, \Z_{4}$ or $D_3$.

Let us explictly describe it. 
If $\omega=\rp{r_1,\ldots,r_5; p_1,\ldots,p_5}$ has nontrivial isotropy group, 
then there are at least two pairs of equal residues or three equal residues. \\
Case \textit{$r_1 = r_2$ and $r_3=r_4$}. 
Since $r_5$ is different from the other residues, 
the pole $p_5$ is a fixed point in the action of $\PSL_{\omega}$ on $\CC$. 
In fact, 
the isotropy group is cyclic. 
If $r_1 \not= r_3$, 
then $\PSL_{\omega} \cong \Z_2$. 
If $r_1 = r_3$, 
then $\PSL_{\omega} \cong \Z_2, \Z_3, \text{ or } \Z_4$. \\
Case \textit{$r_1=r_2=r_3$}. 
We suppose that $\PSL_{\omega}$ is not isomorphic to a cyclic group. 
Since $\PSL_{\omega}$ is nontrivial, 
$r_4=r_5$, and $\{p_4, p_5 \}$ is a orbit of order 2 in the action of $\PSL_{\omega}$ on $\CC$. 
In fact, the isotropy group is dihedral. 
Since $r_1=r_2=r_3$, 
the isotropy group $\PSL_{\omega} \cong D_3$.
\end{upshape}
\end{example}


Numerical conditions on $s \geq 3$, 
to realize $G < \PSL$ as an isotropy group for some $\omega \in \racs{s}$, 
are as follow. 
 
\begin{proposition}
\label{prop-numerica}
Consider $n \geq 2$ and $n_1, n_2 \in \N\cup\{ 0 \}$ such that $n_1 + n_2 \geq 2$. 
There exists $\omega \in \racs{s}$ such that
\begin{enumerate}
\item $\PSL_{\omega} \cong \Z_n$ if and only if $s \equiv 0, 1 \text{ or } 2 \ (\text{mod } n)$, where $s > n$. 

\item $\PSL_{\omega} \cong D_n$ if and only if $s \equiv 0 \text{ or } 2 \ (\text{mod } n)$, where $s > n$. 

\item $\PSL_{\omega} \cong A_4$ if and only if $s = 12n_1 + n_2$, where $n_2 \in \{ 0, 8, 10, 14, 16, 18 \}$.

\item $\PSL_{\omega} \cong \Sim{4}$ if and only if $s = 24n_1 + n_2$, where $n_2 \in \left\{ 0, 14, 18, 20, 26, \right.$ $\left. 30,32, 36 \right\}$.

\item $\PSL_{\omega} \cong A_5$ if and only if $s = 60n_1 + n_2$, where $n_2 \in \left\{ 0, 32, 42, 50, 62,  \right.$ $\left. 72, 80, 90 \right\}$.
\end{enumerate}
\end{proposition}

\begin{proof}
Case 1. 
Consider $\eta = \rp{-ni, i, \ldots, i; \infty, \zeta_1, \ldots, \zeta_n}$, 
where $\zeta_{\iota}$ are the $nth$ roots of unity. 
Obviously, its isotropy group $\PSL_{\eta} = \{e^{2k\pi i / n}z\}\cong \Z_n$. 
For $\omega \in \racs{s}$, 
if $\PSL_{\omega} \cong \Z_n$ then there is $T \in \PSL$ such that $\PSL_{T_*\omega} = T \cdot \PSL_{\omega} \cdot T^{-1} = \PSL_{\eta}$; 
see \cite[p.~107]{diustermaat} and \cite[p.~44]{singerman}. 
It is easy to see that for $p_{\iota} \in \CC$ pole of $T_{*}\omega$, 
its orbit $\PSL_{T_*\omega} \cdot p_{\iota}$, 
under the action of $\PSL_{T_*\omega}$ on $\CC$, 
is a set of poles for $T_*\omega$. 
In other words, 
if $\ell$ is the number of poles $p_{\iota}$ with different orbits, 
then 
$$
s = \# \{\text{poles of }\omega\} = \# \{\text{poles of }T_*\omega\} = \sum_{\iota =1}^{\ell} \#(\PSL_{T_*\omega} \cdot p_{\iota}).
$$
Since $\#(\PSL_{T_*\omega} \cdot p_{\iota}) = n$, for $p_{\iota} \not= 0$ or $\infty$ and $\#(\PSL_{T_*\omega} \cdot 0) = \#(\PSL_{T_*\omega} \cdot \infty ) = 1$, 
the result is proved. 
The other cases are analogous. 
\end{proof}
Using Proposition \ref{prop-numerica}, 
we complete the Table \ref{tab:1}.

\begin{table}[h]
\caption{Finite subgroups of $\PSL$ that appear as isotropy for $\omega \in \racs{s}$.}
\label{tab:1}       
\begin{tabular}{ll}
\hline\noalign{\smallskip}
$s$ & Nontrivial isotropy groups for $\omega \in \racs{s}$  \\
\noalign{\smallskip}\hline\noalign{\smallskip}
3 &  $\Z_2$ \\ 
4 &  $\Z_2,  \Z_3,  \Z_2\times \Z_2$ \\
5 &  $\Z_2,  \Z_3,  \Z_4,  D_3$ \\ 
6 &  $\Z_2, \Z_3, \Z_4, \Z_5, \Z_2 \times \Z_2, D_3, D_4$ \\ 
7 & $\Z_2, \Z_3, \Z_5, \Z_6, D_5$ \\
8 &  $\Z_2, \Z_3, \Z_4, \Z_6, \Z_7, \Z_2 \times \Z_2, D_3, D_4, D_6, A_4$ \\ 
9 &  $\Z_2, \Z_3, \Z_4, \Z_7, \Z_8, D_3, D_7$ \\ 
10 & $\Z_2, \Z_3, \Z_4, \Z_5, \Z_8, \Z_9, \Z_2 \times \Z_2, D_4, D_5, D_8$ \\ 
11 & $\Z_2, \Z_3, \Z_5, \Z_9, \Z_{10}, D_3, D_9$ \\ 
\noalign{\smallskip}\hline
\end{tabular}
\end{table}

\section{Quotients}
\label{S-examplesquotient}

\subsection{Stratification by orbit types}

For $s\geq 3$, 
the $\PSL$--action is proper, 
and the classical theory of Lie groups can be applied. 
Mainly, we follow the theory and notation of J. J. Duistermaat and J. A. Kolk in \cite{diustermaat}. 
In order to describe the quotients $\racs{s}/\PSL$ and $\RI{s}/\PSL$, 
recall that if a Lie group $G$ acts properly on a manifold $M$, 
then every closed subgroup $H$ of $G$ acts in a proper and free way on $G$; 
see \cite[p.~93]{diustermaat}. 
Moreover, 
the right coset $G/G_x$ is a manifold of dimension $dim(G) - dim(G_x)$ diffeomorphic to the orbit $G\cdot x$. \\
In our case, 
since the $\PSL$--action is proper and all isotropy groups are finites, 
the orbits $\PSL \cdot \omega$ under $\mathcal{A}_s$ are 3--dimensional complex submanifolds of $\racs{s}$, 
biholomorphic to the right coset ${\PSL}/{\ \PSL_{\omega}}$. 
Similarly, 
for $\omega \in \RI{s}$ its orbit $\PSL \cdot \omega$ is a 6--dimensional real analytic submanifold of $\RI{s}$, 
and $\PSL \cdot \omega$ is diffeomorphic to the right coset ${\PSL}/{\ \PSL_{\omega}}$.   \\
Furthermore, 
by applying Theorem 2.7.4 in \cite{diustermaat} the quotients $\racs{s}/\PSL$ and ${\RI{s}}/{\PSL}$ admit a stratification by orbit types. 
The action $\mathcal{A}_s$ is proper and free in the generic open and dense subset
\begin{equation}\label{generic-forms}
\mathcal{G}(-1^s) := \{\omega \in \racs{s} \ | \ \PSL_{\omega} \cong \{Id \} \}.
\end{equation}
\noindent The quotient $\mathcal{E}(-1^s) := {\mathcal{G}(-1^s)}/{\PSL}$ is a $(2s-4)$--dimensional complex manifold. 
By following \cite[p.~107]{diustermaat}, 
for each $\omega \in \racs{s}$ its orbit type is
$$
\racs{s}^{\sim}_{\omega} := \left\{ \eta \in \racs{s} \ | \PSL_{\eta} \cong \PSL_{\omega} \right\}.
$$
Note that in our case the isotropy groups are isomorphic in $\racs{s}^{\sim}_{\omega}$ instead of conjugates since for finites subgroups of $\PSL$, they are equivalents; see \cite[p.~50]{singerman}.
Similarly, its orbit type on the quotient is $\racs{s}^{\sim}_{\omega}/\PSL$. 
Looking at $\racs{s}^{\sim}_{\omega}$, its connected components $\{E_j\}$ are the stratum and they are complex submanifolds of $\racs{s}$ with dimension $dim(E_j) \leq 2s-1$. 
The higher dimensional stratum is $\mathcal{G}(-1^s)$. 
For the quotient, 
the connected components of $\racs{s}^{\sim}_{\omega} / \PSL $ are the stratum and they are complex manifolds with dimension less or equal to $2s-4$. 
The higher dimensional stratum is $\mathcal{E}(-1^s)$.  

\begin{remark}
There exists a holomorphic principal $\PSL$--bundle
$$
\xymatrix{\PSL \ar@{->}[r] &  \mathcal{G}(-1^s) \ar@{->}^-{\pi_s}[d] \\
                                        &  \mathcal{E}(-1^s) \ ,}
$$
where $\pi_s$ denotes the natural projection to the $\PSL$--orbits. 
\end{remark}

For $\RI{s}$, the generic open and dense real analytic submanifold is
$$
\mathcal{RIG}(-1^s) :=  \{\omega \in \RI{s} \ | \ \PSL_{\omega} \cong \{Id \} \}.
$$
The quotient $\mathcal{RIE}(-1^s) := \mathcal{RIG}(-1^s)/\PSL$ is a $(3s-7)$--dimensional real analytic manifold. 
For $\RI{s}$ and $\RI{s}/\PSL$, their stratification by orbit types are analogous as for $\racs{s}$ and $\racs{s}/\PSL$, respectively. \\

\subsection{Realizations}

Let us define a realization\footnote{We use definition of realization as in \cite[p.~6]{yoshida}.} for the quotient $\racs{s}/\PSL$ by using a complete set of $\PSL$--invariant functions. 
First, we consider the ordered set of residues as the complement of an arrangement of $s$ hyperplanes
$$
\mathbb{A}_s = \C_{(r_1, \ldots, r_{s-1})}^{s-1} \setminus \left\{r_1+\ldots + r_{s-1}=0, \ r_{\iota}=0 \ \iota=1,\ldots, s-1 \right\}.
$$
For $s=2,3$, 
the residues are a complete set of $\PSL$--invariant functions.

\begin{example}
\begin{upshape}
Case $s=2$, 
the natural projection 

\centerline{$\pi_2 : \racs{2} \longrightarrow {\racs{2}}/{\PSL}: \ \rp{r_1, r_2; p_1, p_2} \longmapsto r_1$}

\noindent determines a fiber bundle. 
Obviously, the base space is biholomorphic to $\C^*/\Z_2$. 
Similarly, the quotient $\RI{2}/\PSL$ is diffeomorphic to $\R^{+} = \{r_1 \ | \ r_1 > 0 \}$. 
For both cases, the fibers are $\CC^2 \setminus \Delta = \{(p_1, p_2) \ | \ p_1 \not= p_2 \}$. 
\end{upshape}
\end{example}

\begin{example}\label{cociente-3}
\begin{upshape}
	Case $s=3$, using Example \ref{ejemplo-S2}.2 the quotient $\racs{3}/\PSL$ admits a stratification with two orbit types. 
It is homemorphic to $\mathbb{A}_3/\Sim{3}$, where the symetric group $\Sim{3}$ acts linearly on $\mathbb{A}_3$ using the isomorphism 
$$
\Sim{3} \cong \rp{\left(\begin{smallmatrix}
                   0 & 1 \\
                   1 & 0
                  \end{smallmatrix}\right), \left(\begin{smallmatrix}
     \ \ 1 & \ \ 0 \\
       -1 & -1 
      \end{smallmatrix}\right)}.
$$
Similarly, the quotient ${\RI{3}}/{\PSL}$ has two connected components. 
A fundamental domain is
$$
\{ (r_1, r_2) \ | \ r_1r_2 >0  \text{ and } r_1 \leq r_2 \} \subset \mathbb{A}_3.
$$
Here $(r_1,r_2)$ determines the 1--form $\rp{ir_1, ir_2, -i(r_1+r_2); 0, \infty, 1}$. 
Their connected components come from $\{r_1 > 0\}$ and $\{r_1 < 0\}$. 
The orbit types are $\{r_1=r_2\}$ and $\{r_1 < r_2\}$. 
Since the connected components are contractibles, 
the corresponding principal $\PSL$--bundle $\pi_3: \mathcal{RIG}(-3) \longrightarrow \mathcal{RIE}(-3)$ is trivial.
\end{upshape}
\end{example}

For $s \geq 4$, the residues are not a complete set of $\PSL$--invariant functions. 
In order to enlarge our set, 
we fix three poles in $\{0, \infty, 1\}$ and consider the ordered set of poles with 
$$
\left[ \C^* \setminus \{ 1 \}\right]^{s-3}\setminus \Delta := \left. \left\{(p_4, \ldots, p_s) \in \left[ \C^* \setminus \{ 1 \}\right]^{s-3} \ \right| \ p_{\iota} \not= p_{\kappa} \text{ for } \iota \not= \kappa \right\} . 
$$

Given a configuration $\{ q_1, \ldots, q_s \} \subset \CC$,  
there exist $\left(\begin{smallmatrix} s \\  s-3 \end{smallmatrix}\right)3!$ M\"obius transformations $T \in \PSL$ such that $\{T(q_1), \ldots, T(q_s) \} = \{0, \infty, 1, p_4, \ldots, p_s \}$.

\begin{remark}
For an ordered collection

\centerline{$(r_1, \ldots, r_{s-1},p_4, \ldots, p_s) \in \mathbb{A}_s \times \left[ \C^* \setminus \{ 1 \}\right]^{s-3}\setminus \Delta : = \mathcal{M}(-s)$,}

\noindent and each permutation $\sigma \in \Sim{s}$, there exists a unique $T_{\sigma} \in \PSL$ that 
$$
\begin{array}{rl}
(r_1, \ldots, r_s, p_4, \ldots, p_s)  & \longmapsto \\
&\vspace{-0.2cm}\\
&\hspace{-2cm} \left\{(r_{\sigma(1)}, 0), (r_{\sigma(2)}, \infty), (r_{\sigma(3)}, 1), (r_{\sigma(4)}, T_{\sigma}(p_4)), \ldots, (r_{\sigma(s)}, T_{\sigma}(p_s))\right\} \\
&\vspace{-0.2cm}\\
&\hspace{-2cm} := \rp{\rp{ r_1, \ldots, r_s; 0, \infty, 1, p_4, \ldots, p_s }} \in \racs{s}/\PSL.
\end{array}
$$ 
\end{remark}

Note the appearance of $r_s= -(r_1+ \ldots + r_{s-1})$ and $0, \infty, 1$ on the right side. 
There is a natural $\Sim{s}$--action on $\mathcal{M}(-s)$. 
In order to recognize it, 
we define a group representation in the Coexeter generators of $\Sim{s}$, 
see \cite[Sec.~1.2]{coxeter}, 
as
$$
\begin{array}{rcl}
\rho_s : \Sim{s} & \longrightarrow & GL_{s-1}( \Z) \times Bir(\CC^{s-3}) \\
&&\vspace{-0.2cm}\\
\sigma_j = (j \ j+1) & \longmapsto & {(A_j, f_j) = \left\{ \begin{array}{ll}
                          \left(A_1, \left( \frac{1}{z_4}, \ldots, \frac{1}{z_s} \right) \right) &\\
&\vspace{-0.2cm}\\
 \left(A_2, \left( \frac{z_4}{z_4-1}, \ldots, \frac{z_s}{z_s-1} \right) \right) &\\
& \vspace{-0.2cm}\\
 \left(A_3, \left( \frac{1}{z_4}, \frac{z_5}{z_4}, \ldots, \frac{z_s}{z_4} \right) \right) &\\
& \vspace{-0.2cm}\\
 \left(A_{j}, \left( z_{\sigma_{\iota}(4)}, \ldots, z_{\sigma_{j}(s)}\right) \right),&\\
&\hspace{-1.2cm} \text{where } j = 4, \ldots, s-1.
                                  \end{array} \right. }
\end{array}
$$
Here $Bir(\CC^{s-3})$ denotes the group of complex birational maps on $\CC^{s-3}$; 
the birational map $f_1$ from $\sigma_1$ must be understood as $f_{1}: (z_4, \ldots, z_s) \longmapsto  ( {1}/{z_4}, \ldots, {1}/{z_s} )$. 
Since there is a biholomorphism between the Torelli space of the $s$--punctured sphere and $[\C^* \setminus\{1 \}]^{s-3} \setminus \Delta$, 
the subgroup of birational maps $\{ f_{\sigma} \}$ is the corresponding Torelli modular group; 
see \cite{patterson}. \\
For $j=1, \ldots, s-2$, the matrices $A_j$ come from the identity matrix by exchanging the $j\text{th}$--row with the $(j+1)\text{th}$--row; 
for $j=s-1$, $A_{s-1}$ results from replacing, 
in the identity matrix, 
the $(s-1){\text{th}}$--row with $(-1, \ldots, -1)$. \\
It is a straighforward computation that $\{\rho_s(\sigma_{j})\}$ satisfy the relations in Coxeter's presentation. 
By using $\rho_s$, we define a $\Sim{s}$--action on $\mathcal{M}(-s)$ as
\begin{equation}\label{S-action}
\begin{array}{rcl}
\Sim{s} \times \mathcal{M}(-s) & \longrightarrow & \mathcal{M}(-s)\\
(\sigma, (r_1, \ldots, r_{s-1}, p_4, \ldots, p_s)) & \longmapsto & \left({A_{\sigma}}\left(\begin{smallmatrix}
                   r_1 \\
                    \vdots \\
                   r_{s-1}
                  \end{smallmatrix}\right),\  f_{\sigma}(p_4, \ldots, p_s)\right).
\end{array}
\end{equation}
In order to recognize the quotient $\racs{s}/\PSL$, the map
$$
\begin{array}{rcl}
\mu_s: \mathcal{M}(-s) & \longrightarrow & \racs{s}\\
(r_1, \ldots, r_{s-1}, p_4, \ldots, p_s) & \longmapsto & \rp{r_1, \ldots, r_{s}; 0, \infty, 1, p_4, \ldots, p_s} = \omega,
\end{array}
$$
will be useful. 
The number of preimages $ \mu_s^{-1}(\omega)$ is $(s-3)!$ 
Furthermore, 
the number of preimages $(\pi_s \circ \mu_s)^{-1}\rp{\rp{\omega}}$ is less than or equal to $s!$ and the equality is fulfilled when $\omega \in \mathcal{G}(-1^s)$; 
recall \eqref{generic-forms}.

\begin{proposition}\label{quotient}
	For $s\geq 4$, the realization of the quotient $\racs{s}/\PSL$ is $\mathcal{M}(-s)/\Sim{s}$.
\end{proposition}

\begin{proof}
Let us prove that $\pi_s \circ \mu_s$ is a $\Sim{s}$--equivariant map, \textit{i. e.} 

\centerline{$(\pi_s \circ \mu_s)(\sigma \cdot (r_1, \ldots, r_{s-1}, p_4, \ldots, p_s)) =  (\pi_s \circ \mu_s)(r_1, \ldots, r_{s-1}, p_4, \ldots, p_s)$,}

\noindent for all $\sigma \in \Sim{s}$. 
For example, consider $\sigma_{1} = (1 \ 2) \in \Sim{s}$, the explicit calculation is
$$
\begin{array}{rl}
(\pi_s \circ \mu_s)(\sigma_1 \cdot (r_1, \ldots, r_{s-1}, p_4, \ldots, p_s)) &\\
&\hspace{-1.7cm}= (\pi_s \circ \mu_s)(r_2, r_1, r_3, \ldots, r_{s-1}, 1/p_4, \ldots, 1/p_s)\\
&\hspace{-1.7cm}= \rp{\rp{ r_2, r_1, r_3, \ldots, r_s;0, \infty, 1, 1/p_4, \ldots, 1/p_s }} \\
&\hspace{-1.7cm}= \rp{\rp{(1/z)_*\rp{r_2, r_1, r_3, \ldots, r_s; \infty, 0, 1, p_4, \ldots, p_s} }}\\
&\hspace{-1.7cm}= \rp{\rp{r_1, \ldots, r_s; 0, \infty, 1, p_4, \ldots, p_s }}\\
&\hspace{-1.7cm}= (\pi_s \circ \mu_s)(r_1, \ldots, r_{s-1}, p_4, \ldots, p_s).
\end{array}
$$
On the other hand, 
$(\pi_s \circ \mu_s)$ is surjective. 
Therefore, 
there exists a homeomorphism $\tilde{\mu}_s: \mathcal{M}(-s)/\Sim{s} \longrightarrow \racs{s}/\PSL$ such that the diagram below conmutes. 
$$
\xymatrix{& \mathcal{M}(-s) \ar@{->}[d] \ar@{->}^-{\pi_s \circ \mu_s}[d] \ar@{->}[dl]\\
    \frac{\mathcal{M}(-s)}{\Sim{s}}  \ar@{-->}_-{\tilde{\mu}_s}[r]&   \frac{\racs{s}}{\PSL}.}
$$

\end{proof}

Similarly, we define $\Im{\mathbb{A}_s} := \{(ir_1, \ldots, ir_{s-1}) \in \mathbb{A}_s \ | \ r_{\iota} \in \R^*, \ \iota=1,\ldots, s-1 \}$. 
Since the $\Sim{s}$--action \eqref{S-action} is well--defined on $\Im{\mathcal{M}(-s)} := \Im{\mathbb{A}_s} \times [\C^*\setminus \{1\}]^{s-3} \setminus \Delta$, 
The result below was proved.
\begin{corollary}\label{quotient2}
	For $s\geq 4$, the realization of the quotient $\RI{s}/\PSL$ is $\Im{\mathcal{M}(-s)}/\Sim{s}$.
\end{corollary}
For $s=4$, 
the number of connected components of $\Im{\mathcal{M}(-4)}$ depends only on the number of connected components of $\Im{\mathbb{A}_4}$. 
In fact, $\Im{\mathcal{M}(-4)}$ has 14 connected components,
$$
	\begin{array}{lcl}
		X_{j}^{+} &:=& \left\{(ir_1,ir_2,ir_3, p_4) \in \Im{\mathcal{M}(-4)}  \ \left| \ r_4 > 0, \ r_{j} > 0, \ r_{\iota} < 0 \ \iota \not= j\right\}\right. ,\\
		X_{j}^{-} &:=& \left\{(ir_1,ir_2,ir_3, p_4) \in \Im{\mathcal{M}(-4)} \ \left| \ r_4 < 0, \ r_{j} > 0, \ r_{\iota} < 0 \ \iota \not= j
\right\}\right. ,\\
		X_{j_1j_2}^{+} &:=& \left\{(ir_1,ir_2,ir_3, p_4) \in \Im{\mathcal{M}(-4)} \ \left| \ r_4 > 0, \ r_{j_1} > 0, r_{j_2} > 0 \ r_{j_3} < 0 \right\}\right. ,\\
		X_{j_1j_2}^{-} &:=& \left\{(ir_1,ir_2,ir_3, p_4) \in \Im{\mathcal{M}(-4)} \ \left| \ r_4 < 0, \ r_{j_1} > 0, r_{j_2} > 0 \ r_{j_3} < 0 \right\}\right. ,\\
		X_{+} &:=& \left\{(ir_1, ir_2,ir_3, p_4) \in \Im{\mathcal{M}(-4)} \ | \ r_j > 0, \ \ j=1,2,3\right\} ,\\
		X_{-} &:=& \left\{(ir_1, ir_2,ir_3, p_4) \in \Im{\mathcal{M}(-4)} \ | \ r_j < 0, \ \ j=1,2,3\right\} .\\
	\end{array}
$$
By applying the $\Sim{4}$--action \eqref{S-action}, 
these components are identified as
$$
\begin{array}{c}
X_+ \sim X^{+}_{23} \sim X^{+}_{13} \sim X^{+}_{12}, \\
X_- \sim X^{-}_{1} \sim X^{-}_{2} \sim X^{-}_{3}, \\
X^{+}_{1} \sim X^{+}_{2} \sim X^{+}_{3} \sim X^{-}_{23} \sim X^{-}_{13} \sim X^{-}_{12}.
\end{array}
$$
Using Proposition \ref{quotient}, Corollary \ref{quotient2} and Table \ref{tab:1}, 
result below was proved.
\begin{lemma}
1. The quotient $\racs{4}/\PSL$ is connected and it admits a stratification with 4 orbit types and 5 stratum.\\
2. The quotient ${\RI{4}}/{\PSL}$ admits a stratification with 4 orbit types, 3 connected components and 10 stratum. The corresponding principal $\PSL$--bundle $\pi_4: \mathcal{RIG}(-4) \longrightarrow \mathcal{RIE}(-4)$ is nontrivial.
\end{lemma}

\begin{remark}
1. The complex manifold $\racs{s}$ and $\racs{s}/\PSL$ are connected.\\
2. The real analityc manifold $\RI{s}$ and the quotient $\RI{s}/\PSL$ have $s-1$ connected components.
\end{remark}
For odd numbers $3 \leq s \leq 11$, the Table \ref{tab:2} shows the number of orbit types and stratum on the quotient $\RI{s}/\PSL$.
\begin{table}[h]
\caption{}
\label{tab:2}       
\begin{tabular}{lll}
\hline\noalign{\smallskip}
 s & \text{Orbit types} & \text{Stratum} \\
\noalign{\smallskip}\hline\noalign{\smallskip}
3 & 2 & 4 \\ 
5 & 5 & 16 \\ 
7 & 6 & 24 \\ 
9 & 8 &  32 \\ 
11 & 8 & 48 \\ 
\noalign{\smallskip}\hline
\end{tabular}
\end{table}

\section{The associated singular flat surfaces $S_{\omega}$}
\label{sec-surfaces}

\subsection{Isometries}

Recall that for $\omega \in \racs{s}$, 
there is a complex atlas $\{(V_j, \Psi_j )\}$ on $X_{\omega} = \CC \setminus \{\text{zeros and poles of } \omega \}$, 
where $\{ V_j \}$ is an open cover by simply connected sets and the functions 
$$
\Psi_{j}(z) = \int_{z_0}^{z}{\omega}: V_j \longrightarrow \C
$$
are well--defined for all $j$. 
Moreover, 
$\Psi_{jk}(z) = z + a_{jk}$, for $a_{jk} \in \C$. 
If $\omega \in \Omega^1\{k_1, \ldots, k_m; -1, \ldots, -1\} \subset \racs{s}$, 
then the zero of multiplicity $k_j$ is a singularity of cone angle $(2k_j+2)\pi$ and the pole $p_{\iota}$ is a cylindrical end of diameter $T_{\iota} = 2\pi |r_{\iota}|$, 
where $r_{\iota} = Res(\omega, p_{\iota})$, 
$j = 1, \ldots, m$ and $\iota =1, \ldots , s$; 
see \cite{jesus1, valero}. 

For $\omega=({Q(z)}/{P(z)})dz \in \racs{s}$, 
its associated singular flat surface $S_{\omega} = (\CC, g_{\omega})$ 
has the riemannian metric
$$
g_{\omega}(z) := \left(\begin{smallmatrix}
               \left| \frac{Q(z)}{P(z)} \right|^2 & 0 \\
                0 & \left| \frac{Q(z)}{P(z)} \right|^2
             \end{smallmatrix}\right).
$$
We denote by $S^1$ the unit circle on $\C$. 
The result below is well--known; 
see \cite{valero}. 

\begin{proposition}
\label{isometric-surfaces}
 For $\omega, \eta \in \racs{s}$, their associated singular flat surfaces $S_{\omega}$ and $S_{\eta}$ are isometric if and only if there exist $\lambda \in S^1$ and $T \in \PSL$ such that $\eta = \lambda T_*\omega$.
\end{proposition}

\subsection{The $(S^1 \times \PSL)$--action}

By applying Proposition \ref{isometric-surfaces}, 
we can extend naturally the $\PSL$--action \eqref{PSL-action} as follows.
\begin{equation}\label{PSL-action2}
\begin{array}{rcl}
\widehat{\mathcal{A}}_s: (S^1 \times \PSL) \times \racs{s} & \longrightarrow & \racs{s} \\
((\lambda, T), \omega) & \longmapsto & \lambda T_*\omega.
\end{array}
\end{equation}
Similarly, 
we can extend the $\Sim{s}$--action \eqref{S-action} as
\begin{equation}\label{S-action2}
\begin{array}{rcl}
(S^1 \times \Sim{s}) \times \mathcal{M}(-s) & \longrightarrow & \mathcal{M}(-s) \\
((\lambda, \sigma), (r_1, \ldots, r_{s-1}, p_4, \ldots, p_s)) & \longmapsto & \left(\lambda{A_{\sigma}}\left(\begin{smallmatrix}
                   r_1 \\
                    \vdots \\
                   r_{s-1}
                  \end{smallmatrix}\right),\  f_{\sigma}(p_4, \ldots, p_s)\right).
\end{array}
\end{equation}
The expression for the $(S^1 \times \PSL)$--action, 
using the complex atlas by residues--poles on $\racs{s}$, 
is 
$$
\widehat{\mathcal{A}}_s(\lambda, T, \rp{r_1, \ldots, r_s; p_1, \ldots, p_s }) = \rp{\lambda r_1,\ldots, \lambda r_s; T(p_1), \ldots, T(p_s)}.
$$
We use the techniques developed in Section \ref{Sec:action} and \ref{S-examplesquotient} to prove the results below. 

\begin{remark} 
Since the $(S^1 \times \PSL)$--action is proper, 
the quotient
$$
\frac{\racs{s}}{S^1 \times \PSL} = \frac{\{S_{\omega} \ | \ \omega \in \racs{s} \}}{\{\text{Isometries}\}} := \mathfrak{M}(-s)
$$
admits a stratification by orbit types. 
Furthermore, 
the realization for the quotient $\mathfrak{M}(-s)$ is $\mathcal{M}(-s)/ S^1 \times \Sim{s}$.
\end{remark}
For the subgroups $\Z_2 \times \PSL < S^1 \times \PSL$ and $\Z_2 \times \Sim{s} < S^1 \times \Sim{s}$, 
the actions \eqref{PSL-action2} and \eqref{S-action2} are well--defined on $\RI{s}$ and $\Im{\mathcal{M}(-s)}$, 
respectively. 
\begin{remark} 
The $(\Z_2 \times \PSL)$--action on $\RI{s}$ is proper, 
therefore the quotient
$$
\frac{\RI{s}}{\Z_2 \times \PSL} = \frac{\{S_{\omega} \ | \ \omega \in \RI{s} \}}{\{\text{Isometries}\}} := \mathcal{RI}\mathfrak{M}(-s)
$$
admits a stratification by orbit types. 
Furthermore, 
the realization for the quotient $\mathcal{RI}\mathfrak{M}(-s)$ is homeomorphic to $\Im{\mathcal{M}(-s)}/ \Z_2 \times \Sim{s}$.
\end{remark}

\begin{example}
\begin{upshape}
The quotient $\mathfrak{M}(-3)$ and $\mathcal{RI}\mathfrak{M}(-3)$ are connected and their admit a stratification with two orbit types. 
For $\mathcal{RI}\mathfrak{M}(-3)$, 
a fundamental domain is 
$$
\{(r_1, r_2) \ | \ 0 < r_1 \leq r_2 \},
$$
and the orbit types are $\{r_1 = r_2 \}$ and $\{ r_1 < r_2 \}$.
\end{upshape}
\end{example}










\subsection*{\small{Acknowledgements}}
\small{The author would like to thank his advisor Jes\'us Muci\~no--Raymundo for all fruitful discussions with him during the preparation of this paper and the PhD thesis. 
This work was supported by a PhD scholarship provided by CONACyT at the Centro de Ciencias Matem\'aticas, UNAM and Instituto de F\'isica y Matem\'aticas de la Universidad Michoacana de San Nicol\'as de Hidalgo. }









\bibliographystyle{plain} 
\bibliography{referencias}
\end{document}